\documentclass[12pt]{article}
\usepackage{amssymb,amsmath}
\usepackage{amsthm}
\usepackage{cases}
\usepackage{amsfonts}
\usepackage{graphicx}
\usepackage{cite,color,xcolor}
\usepackage[left=2.6cm,right=2.6cm,top=2.3cm,bottom=2.3cm]{geometry}
\usepackage[colorlinks,citecolor=blue,urlcolor=blue]{hyperref}
\usepackage[utf8]{inputenc}

\newtheorem{theorem}{Theorem}[section]

\newtheorem{lemma}{Lemma}[section]
\newtheorem{proposition}{Proposition}[section]

\newtheorem{definition}{Definition}[section]
\newtheorem{remark}{Remark}[section]

\usepackage{txfonts}
\newcommand{\bal}{\begin{align}}
\newcommand{\bbal}{\begin{align*}}
\newcommand{\beq}{\begin{equation}}
\newcommand{\eeq}{\end{equation}}
\newcommand{\bca}{\begin{cases}}
\newcommand{\eca}{\end{cases}}

\newcommand{\pa}{\partial}
\newcommand{\fr}{\frac}

\newcommand{\De}{\Delta}

\newcommand{\ep}{\varepsilon}
\newcommand{\dd}{\mathrm{d}}

\newcommand{\R}{\mathbb{R}}

\newcommand{\Z}{\mathbb{Z}}

\newcommand{\les}{\lesssim}

\newcommand{\f}{\left}
\newcommand{\g}{\right}
\newcommand{\bi}{\Big}
\begin{document}
\bibliographystyle{plain}
\title{A remark on the zero-filter limit for the Camassa-Holm equation in $B^s_{2,\infty}(\R)$ }

\author{Guorong Qu$^{1}$, Jianzhong Lu$^{2,}$\footnote{E-mail: guorongqu@163.com; louisecan@163.com(Corresponding author); dwhh0406@163.com
}, Wei Deng$^{3}$\\
\small $^1$ School of Tourism Data, Guilin Tourism University, Guilin 541006, China\\
\small $^2$ School of Mathematics and Information, Xiangnan University, Chenzhou, 423000, China\\
\small $^3$ Department of Mathematics, Ganzhou Teachers College, Ganzhou 341000, China
}

\date{\today}
\maketitle\noindent{\hrulefill}

{\bf Abstract:}
This paper investigates the zero-filter limit problem associated with the Camassa-Holm equation. In the work cited as \cite{C.L.L.W.L}, it was established that, under the hypothesis of initial data $u_0\in B^s_{2,r}(\R)$ with $s>\frac32$ and $1\leq r<\infty$, the solutions $\mathbf{S}_{t}^{\mathbf{\alpha}}(u_0)$ of the Camassa-Holm equation exhibit convergence in the $L^\infty_T(B^s_{2,r})$ norm to the unique solution of the Burgers equation as $\alpha\rightarrow 0$. Contrary to this result, the present study demonstrates that for initial data $u_0\in B^s_{2,\infty}(\R)$ the solutions of the Camassa-Holm equation fail to converge strongly in the $L^\infty_T(B^s_{2,\infty})$ norm to the Burgers equation as $\alpha\rightarrow 0$.




\vskip0mm\noindent{\hrulefill}

\section{Introduction}

In this work, we investigate the Camassa-Holm (CH) equation, which is expressed as
\begin{align}\label{alpha-c}
\begin{cases}
u_t+u \partial_xu=-\partial_x \left(1-\alpha^2 \partial_x^2\right)^{-1}\left(u^2+\frac{\alpha^2}{2} (\partial_xu)^2\right), \\
u(0,x)=u_0(x),
\end{cases}
\end{align}
and can be equivalently reformulated as
\begin{align}\label{alpha-cr}
\begin{cases}
u_t+3u \partial_xu=-\alpha^2\partial^3_x \left(1-\alpha^2 \partial_x^2\right)^{-1}(u^2)-\frac{\alpha^2}{2}\partial_x \left(1-\alpha^2 \partial_x^2\right)^{-1} (\partial_xu)^2,\\
u(0,x)=u_0(x).
\end{cases}
\end{align}

When the parameter $\alpha$ vanishes, the CH equation simplifies to the Burgers equation
\begin{align}\label{b}
\begin{cases}
u_t+3u \partial_xu=0,\\
u(0,x)=u_0(x),
\end{cases}
\end{align}
which can be viewed as the non-viscous counterpart ($\nu=0$) of the well-known Burgers equation:
$$\partial_tu+u \partial_xu-\nu \partial_{xx}u=0.$$
Proposed by Burgers in the 1940s as a simplified model for turbulence, this equation emulates the Navier-Stokes equation through its representation of nonlinear advection and diffusion, albeit in a one-dimensional context without a pressure gradient to drive the flow. For further details, refer to \cite{miao2009, Molinet}. The Burgers equation has been extensively studied since its inception and serves as a fundamental example of a PDE that evolves to produce shocks.

The CH equation was initially introduced in the study of hereditary symmetries as an integrable system \cite{Fokas}. It was later rediscovered by Camassa and Holm \cite{Camassa} as a model for shallow water wave dynamics, subsequently becoming widely recognized as the Camassa-Holm equation. The equation is completely integrable, possessing a bi-Hamiltonian structure \cite{Constantin-E,Fokas} and an infinite number of conservation laws \cite{Camassa,Fokas}. Among its most notable features are the peakon solutions, which take the form $ce^{-|x-ct|}$ with $c>0$, and have garnered significant interest in the field of physics \cite{Constantin-I,t}. Another striking characteristic of the CH equation is the phenomenon of wave breaking, where the solution remains bounded while its slope becomes unbounded within a finite time frame \cite{Constantin,Escher2,Escher3}. The existence of global weak solutions and dissipative solutions has been extensively studied \cite{bc1,bc2,xin}, with further results available in the cited references.


Recent literature has focused on the well-posedness of the Camassa-Holm equation in Sobolev and Besov spaces. Li and Olver \cite{wp19,wp23} demonstrated that the Cauchy problem (1.1) is locally well-posed for initial data $u_0\in H^{s}(\R)$ with $s>3/2$. Danchin \cite{wp13} examined local well-posedness in Besov spaces, establishing well-posedness in $B^s_{p,r}(\R)$ for $s>\max\{1+1/p,3/2\}$ and $1\leq p\leq\infty,~1\leq r<\infty$, with the continuous dependence aspect demonstrated by Li and Yin \cite{wp18}. For the critical case $s=3/2$, Danchin \cite{wp14} confirmed local well-posedness in $B^{3/2}_{2,1}(\R)$ and ill-posedness in $B^{3/2}_{2,\infty}(\R)$. In the critical space $B^{3/2}_{2,2}(\R)=H^{3/2}(\R)$, Guo et al. \cite{glmy} demonstrated norm inflation and consequent ill-posedness in $B^{1+1/p}_{p,r}(\R)$ for $1\leq p\leq\infty,~1< r\leq\infty$ using peakon solutions, thereby resolving the open problem posed by Danchin \cite{wp14}.

The investigation into the zero-filter limit of the Camassa-Holm equation was initiated by Gui and Liu in \cite{GL}, who demonstrated that the solutions of the dissipative Camassa-Holm equation converge, at least locally, to those of the dissipative Burgers equation as the filtering parameter $\alpha$ approaches zero within Sobolev spaces of lower regularity. More recently, Li et al. \cite{lyz} established that, as $\alpha\to0$, the solutions of the Camassa-Holm equation with fractional dissipation exhibit strong convergence in the space $L^\infty(0,T;H^s(\R))$ toward the inviscid Burgers equation for any initial data $u_0\in H^s(\R)$ with $s>3/2$ and for some $T>0$. In contrast, the authors of \cite{lyz1} further revealed that, for initial data $u_0\in H^s(\R)$ with $s>3/2$ and for some $T>0$, the solutions of the CH equation fail to converge uniformly with respect to the initial data in $L^\infty(0,T;H^s(\R))$  to the inviscid Burgers equation as $\alpha\to0$. In a recent publication \cite{C.L.L.W.L}, these findings were extended to more general Besov spaces $B^s_{2,r}(\R)$ with $s>\frac32$ and $1\leq r<\infty$. Aligning with this research direction, the aim of this work is to investigate the zero-filter limit for the Camassa-Holm equation with initial data $u_0$ belonging to the Besov spaces $B^s_{2,\infty}(\R)$. The primary result of this study is formulated as follows:


\begin{theorem}\label{th1} Assume that $s>\frac32$ and $\alpha\in[0,1]$. Let $\mathbf{S}_{t}^{\mathbf{\alpha}}(v_0)$  be the solutions of \eqref{alpha-c} with the initial data $v_0\in B^s_{2,\infty}(\R)$. Then there exist a  initial data $u_0\in B^s_{2,\infty}(\R)$ and a positive constant $\eta_0$ such that
$$\left\|\mathbf{S}_{t_n}^{\alpha_n}(u_0)-\mathbf{S}_{t_n}^{0}(u_0)\right\|_{
B^s_{2,\infty}}\geq \eta_0,$$
with  $\alpha_n\rightarrow 0$ and $t_n\rightarrow 0$.
\end{theorem}

It should be noted that this finding underscores the critical dependence of solution convergence on the regularity and structure of the initial data, thereby contributing to a deeper understanding of the limiting behavior of solutions to nonlinear partial differential equations.

\section{Preliminaries}\label{sec2}
{\bf Notation}\; Throughout this paper, the symbol $C$ is employed to represent a generic positive constant that remains independent of the parameter $\alpha$. This constant may vary from line to line. For a Banach space $X$, its associated norm is denoted by $\|\cdot\|_{X}$. When considering an interval $I\subset\R$, the space of continuous functions on $I$ taking values in $X$ is denoted by $\mathcal{C}(I;X)$. For notational simplicity, the space $L^p(0,T;X)$  is occasionally abbreviated as $L_T^pX$.
For any tempered distribution $f\in \mathcal{S}'$, the Fourier transform $\mathcal{F}(f)$ and its inverse $\mathcal{F}^{-1}(f)$ are respectively defined as
$$\mathcal{F}(f)(\xi)=\int_{\R}e^{-ix\xi}f(x)\dd x \quad\text{and}\quad\mathcal{F}^{-1}(f)(\xi)=
\frac{1}{2\pi}\int_{\R}e^{ix\xi}f(x)\dd x, \quad\; \forall\xi\in\R.$$
Additionally, we revisit essential concepts related to the Littlewood-Paley decomposition, nonhomogeneous Besov spaces, and several pertinent properties that will be utilized in subsequent analysis.

\begin{proposition}[Littlewood-Paley decomposition, See \cite{B.C.D}] Let $\mathcal{B}:=\{\xi\in\mathbb{R}:|\xi|\leq \frac 4 3\}$ and $\mathcal{C}:=\{\xi\in\mathbb{R}:\frac 3 4\leq|\xi|\leq \frac 8 3\}.$
There exist two radial functions $\chi\in C_c^{\infty}(\mathcal{B})$ and $\varphi\in C_c^{\infty}(\mathcal{C})$ both taking values in $[0,1]$ such that
\begin{align*}
&\chi(\xi)+\sum_{j\geq0}\varphi(2^{-j}\xi)=1 \quad \forall \;  \xi\in \R,\\
&\frac{1}{2} \leq \chi^{2}(\xi)+\sum_{j \geq 0} \varphi^{2}(2^{-j} \xi) \leq 1\quad \forall \;  \xi\in \R.
\end{align*}
\end{proposition}

\begin{definition}[See \cite{B.C.D}]
For every $u\in \mathcal{S'}(\mathbb{R})$, the Littlewood-Paley dyadic blocks ${\Delta}_j$ are defined as follows
\begin{numcases}{\Delta_ju=}
0, & if\; $j\leq-2$;\nonumber\\
\chi(D)u=\mathcal{F}^{-1}(\chi \mathcal{F}u), & if\; $j=-1$;\nonumber\\
\varphi(2^{-j}D)u=\mathcal{F}^{-1}\big(\varphi(2^{-j}\cdot)\mathcal{F}u\big), & if\; $j\geq0$.\nonumber
\end{numcases}
In the inhomogeneous case, the following Littlewood-Paley decomposition makes sense
$$
u=\sum_{j\geq-1}{\Delta}_ju,\quad \forall\;u\in \mathcal{S'}(\mathbb{R}).
$$
\end{definition}
\begin{definition}[See \cite{B.C.D}]
Let $s\in\mathbb{R}$ and $(p,q)\in[1, \infty]^2$. The nonhomogeneous Besov space $B^{s}_{p,q}(\R)$ is defined by
\begin{align*}
B^{s}_{p,q}(\R):=\Big\{f\in \mathcal{S}'(\R):\;\|f\|_{B^{s}_{p,q}(\mathbb{R})}<\infty\Big\},
\end{align*}
where
\begin{numcases}{\|f\|_{B^{s}_{p,q}(\mathbb{R})}=}
\left(\sum_{j\geq-1}2^{sjq}\|\Delta_jf\|^r_{L^p(\mathbb{R})}\right)^{\fr1q}, &if\; $1\leq q<\infty$,\nonumber\\
\sup_{j\geq-1}2^{sj}\|\Delta_jf\|_{L^p(\mathbb{R})}, &if\; $q=\infty$.\nonumber
\end{numcases}
\end{definition}
\begin{remark}\label{re3}
It should be emphasized that the following embedding will be used frequently without declaration:
$$B^s_{p,q}(\R)\hookrightarrow B^t_{p,r}(\R)\quad\text{for}\;s>t\quad\text{or}\quad s=t,1\leq q\leq r\leq\infty.$$
\end{remark}
\begin{lemma}[See \cite{B.C.D}]\label{le-pro}
For $s>0$, $B^s_{2,\infty}(\R)\cap L^\infty(\R)$ is an algebra.
Moreover, we have for any $u,v \in B^s_{2,\infty}(\R)\cap L^\infty(\R)$
\begin{align*}
&\|uv\|_{B^s_{2,\infty}(\R)}\leq C\big(\|u\|_{B^s_{2,\infty}(\R)}\|v\|_{L^\infty(\R)}+\|v\|_{B^s_{2,\infty}(\R)}\|u\|_{L^\infty(\R)}\big).
\end{align*}
In particular, for $s>\frac12$, due to the fact $B^s_{2,\infty}(\R)\hookrightarrow L^\infty(\R)$, then we have
\begin{align*}
&\|uv\|_{B^s_{2,\infty}(\R)}\leq C\|u\|_{B^s_{2,\infty}(\R)}\|v\|_{B^s_{2,\infty}(\R)}.
\end{align*}
\end{lemma}
\begin{lemma}[See \cite{B.C.D}]\label{le-pro1}
For $s>\frac32$, there holds
\begin{align*}
&\|uv\|_{B^{s-2}_{2,\infty}(\R)}\leq C\|u\|_{B^{s-2}_{2,\infty}(\R)}\|v\|_{B^{s-1}_{2,\infty}(\R)}.
\end{align*}
\end{lemma}
\begin{lemma}[See \cite{B.C.D}]\label{le3}
Let $s>0$ and $f,g\in {\rm{Lip}(\R)}\cap B^s_{2,\infty}(\R)$. Then we have
\bbal
\big|\big|(2^{js}||[\Delta_j,f]\pa_xg||_{L^2(\mathbb{R})})_{j\geq-1}\big|\big|_{\ell^q}\leq
\begin{cases}
C\big(\|\partial_x f\|_{L^\infty(\R)}\|g\|_{B^{s}_{2,\infty}(\R)}+\|f\|_{B^s_{2,\infty}(\R)}\|\pa_xg\|_{L^\infty(\R)}\big),\\
C\big(\|\partial_x f\|_{L^\infty(\R)}\|g\|_{B^{s}_{2,\infty}(\R)}+\|\pa_x f\|_{B^s_{2,\infty}(\R)}\|g\|_{L^\infty(\R)}\big).
\end{cases}
\end{align*}
\end{lemma}

\begin{lemma}[See \cite{C.L.L.W.L}]\label{lemm1}
For any $\alpha \in(0,1]$, there holds
\bbal
\left|\int_{\R}(1-\alpha^2\pa^2_x)^{-1}\Delta_j(vu_x)\Delta_ju\dd x\right|\leq C\|\pa_xv\|_{L^\infty(\R)}\|\Delta_ju\|^2_{L^2(\R)}+C\|[\Delta_j,v]\pa_xu\|_{L^2(\R)}\|
\Delta_ju\|_{L^2(\R)}.
\end{align*}
\end{lemma}

\begin{lemma}[See \cite{C.L.L.W.L}]\label{lemm2}
For any $\alpha \in(0,1]$, there holds
\bbal
\left|\frac{\alpha^2}{2}\int_{\R}\pa_x(1-\alpha^2\pa^2_x)^{-1}\Delta_j(\partial_xu\pa_xv)\cdot \Delta_jw\dd x\right|\leq C2^{-j}\|\Delta_j(\partial_xu\pa_xv)\|_{L^2}\|\Delta_jw\|_{L^2}.
\end{align*}
\end{lemma}

\begin{lemma}\label{lemm3}
Let $\sigma\in \R$. For any $\alpha \in(0,1]$, we have
\bbal
&\big\|\alpha^2\pa^2_x(1-\alpha^2\pa^2_x)^{-1}u\big\|_{B^\sigma_{2,\infty}}+\big\|\alpha\pa_x(1-\alpha^2
\pa^2_x)^{-1}u\big\|_{B^\sigma_{2,r}}\leq C||u||_{B^\sigma_{2,\infty}},
\\& \big\|\alpha^2\pa_x(1-\alpha^2\pa^2_x)^{-1}u\big\|_{B^\sigma_{2,\infty}}\leq C||u||_{B^{\sigma-1}_{2,\infty}}.
\end{align*}
\end{lemma}
\begin{proof}
The results can be easily deduced from the definition of $\Delta_j$ and Plancherel's identity. Here, we omit the details.
\end{proof}

\section{Proof of the main theorem}

\quad  
In this section, we aim to demonstrate that the zero-filter limit associated with the Camassa-Holm equation fails to converge with respect to the initial conditions specified in the Besov space $B^s_{2,\infty}(\R)$. Initially, we establish that the solution $\mathbf{S}_{t}^{\mathbf{\alpha}}(v_0)$, where $v_0\in B^s_{2,\infty}(\R)$, remains uniformly bounded for all values of the parameter $\alpha$ within the interval $[0,1]$.

\begin{proposition}\label{pro 3.0}
For $s>\frac32$ and $\alpha\in[0,1]$, let $\mathbf{S}_{t}^{\mathbf{\alpha}}(v_0)$ be the solution of \eqref{alpha-c} with the initial data $v_0\in B^s_{2,\infty}(\R)$. Then, there exists a time $T=T(\|v_0\|_{B^s_{2,\infty}},s)>0$ such that $\mathbf{S}_{t}^{\mathbf{\alpha}}(v_0)$ belongs to $L^\infty_T(B_{2, \infty}^s)$ and satisfies
\bbal
||\mathbf{S}_{t}^{\mathbf{\alpha}}(v_0)||_{L^\infty_T(B_{2, \infty}^s)}\leq C||v_0||_{B_{2, \infty}^s}, \quad \forall \alpha \in[0,1].
\end{align*}
\end{proposition}

\begin{proof}
When $\alpha=0$, this result is obvious for Bugers equatuion. Consequently, we focus on establishing the proof for any $\alpha\in(0,1]$. For a fixed $\alpha>0$, by the classical theory of local well-posedness, it is well-established that there exists a $T_\alpha=T(\|v_0\|_{B^s_{2,\infty}},s,\alpha)>0$ such that the Camassa-Holm equation admits a unique solution $\mathbf{S}_{t}^{\mathbf{\alpha}}(v_0)\in\mathcal{C}([0,T_\alpha];B^s_{2,\infty})$.

We aim to demonstrate the existence of a time $T=T(\|v_0\|_{B^s_{2,\infty}},s)>0$ such that $T\leq T_{\alpha}$ for all $\alpha\in(0,1]$. Furthermore, we establish the existence of a constant $C>0$, independent of the parameter $\alpha$, satisfying the following uniform estimate: 
\begin{align}\label{m1}
\big\|\mathbf{S}_{t}^{\mathbf{\alpha}}(v_0)\big\|_{L_T^{\infty} (B^s_{2,\infty})} \leq C||v_0||_{B^s_{2,\infty}}, \quad \forall \alpha \in(0,1].
\end{align}
To simplify notation, we set $v=\mathbf{S}_{t}^{\mathbf{\alpha}}(v_0)$.
Applying the operator $\Delta_j$ to $\eqref{alpha-c}$, multiplying $\Delta_ju$  and integrating the result over $\R$, we obtain
\begin{align}
\frac12\frac{\dd }{\dd t}\|\Delta_jv\|^2_{L^2}&=\frac12\int_{\R}\pa_xv|\Delta_jv|^2\dd x-\int_{\R}[\Delta_j,v]\pa_xv\cdot \Delta_jv\dd x\label{y1}\\
&\quad-2\int_{\R}(1-\alpha^2\pa^2_x)^{-1}\Delta_j(vv_x)\cdot \Delta_jv\dd x\label{y2}\\
&\quad
-\frac{\alpha^2}{2}\int_{\R}\pa_x(1-\alpha^2\pa^2_x)^{-1}\Delta_j(\partial_xv)^2\cdot \Delta_jv\dd x.\label{y3}
\end{align}
To bound \eqref{y1}, it is easy to obtain
\bbal
|\eqref{y1}|
\leq C\|\pa_xv\|_{L^\infty}\|\Delta_ju\|^2_{L^2}+C\|[\Delta_j,v]\pa_x v\|_{L^2}\|\Delta_jv\|_{L^2}.
\end{align*}
To bound \eqref{y2}, by Lemma \ref{lemm1}, we have
\bbal
|\eqref{y2}|\leq C\|\pa_xv\|_{L^\infty}\|\Delta_jv\|^2_{L^2}+C\|[\Delta_j,v]\pa_xv\|_{L^2}\|\Delta_jv\|_{L^2}.
\end{align*}
To bound \eqref{y3}, we can deduce from Lemma \ref{lemm2} that
\bbal
|\eqref{y3}|\leq C2^{-j}\|\Delta_j(\partial_xv)^2\|_{L^2}\|\Delta_jv\|_{L^2}.
\end{align*}
Combining the above estimations yields that
\begin{align*}
\frac{\dd }{\dd t}\|\Delta_jv\|_{L^2}\leq C\Big(\|\pa_xv\|_{L^\infty}\|\Delta_jv\|_{L^2}+\|[\Delta_j,v]\pa_xv\|_{L^2}
+2^{-j}\|\Delta_j(\partial_xv)^2\|_{L^2}\Big),
\end{align*}
which leads to
\bbal
\|\Delta_jv\|_{L^2}\leq \|\Delta_jv_0\|_{L^2}+C\int^t_0\Big(\|\pa_xv\|_{L^\infty}\|\Delta_jv\|_{L^2}+\|[\Delta_j,v]\pa_xv\|_{
L^2}+2^{-j}\big\|\Delta_j(\partial_xv)^2\big\|_{L^2}\Big)\dd \tau.
\end{align*}
Multiplying the above inequality  by $2^{js}$ and taking the $\ell^\infty$ norm over $\Z$, we obtain from Lemma \ref{le-pro} and Lemma \ref{le3} that
\bbal
\|v(t)\|_{B^s_{2,\infty}}&\leq \|(v_0)\|_{B^s_{2,\infty}} + C\int^t_0\Big(\|\pa_xv\|_{L^\infty}\|v\|_{B^s_{2,\infty}}+\big\|(\pa_xv)^2\big\|_{B^{s-1}_{2,\infty}}\Big)\dd \tau
\\&\leq \|(v_0)\|_{B^s_{2,\infty}} + C\int^t_0\|v\|^2_{B^s_{2,\infty}}\dd \tau.
\end{align*}
Consequently, employing a continuity argument, we establish the existence of  $T=T(\|v_0\|_{B^s_{2,\infty}},s)>0$ such that \eqref{m1} holds uniformly w.r.t. $\alpha\in(0,1]$. Thus, we complete the proof of this proposition.
\end{proof}

\begin{proposition}\label{pro 3.1}
For $s>\frac32$ and $\alpha\in[0,1]$, let $\mathbf{S}_{t}^{\mathbf{\alpha}}(v_0)$ be the solution of \eqref{alpha-c} with the initial data $v_0\in B^s_{2,\infty}(\R)$. Then, for $t\in[0,T]$, we have
\bbal
\f\|\mathbf{S}^\alpha_{t}(v_0)-v_0-t\mathbf{E}_0(\alpha,v_0)\g\|_{B_{2,\infty}^{s-2}}\leq Ct^{2},
\end{align*}
where
\bbal
\mathbf{E}_0(\alpha,v_0):=-3v_0\pa_xv_0-\alpha^2\partial^3_x \left(1-\alpha^2 \partial_x^2\right)^{-1}v_0^2-\frac{\alpha^2}{2}\partial_x \left(1-\alpha^2 \partial_x^2\right)^{-1} (\partial_xv_0)^2.
\end{align*}
\end{proposition}

\begin{proof}
For simplicity, we denote $v(t)=\mathbf{S}^\alpha_t(v_0)$.
According to Lemma \ref{le-pro}, Lemma \ref{lemm3} and the fundamental theorem of calculus in the time variable, we obtain that, for $t\in[0,T]$ 
\bal\label{u1}
\|v(t)-v_0\|_{B^{s-1}_{2,\infty}}
&\leq \int^t_0\|\pa_\tau v\|_{B^{s-1}_{2,\infty}} \dd\tau
\nonumber\\&\leq \int^t_0\f(3\|v \pa_xv\|_{B^{s-1}_{2,\infty}}+\alpha^2\f\|\partial^3_x \left(1-\alpha^2 \partial_x^2\right)^{-1}v^2\g\|_{B^{s-1}_{2,\infty}}\g)\dd\tau\nonumber\\&\quad+\int^t_0\frac{\alpha^2}{2}\f\|\partial_x \left(1-\alpha^2 \partial_x^2\right)^{-1} (\partial_xv)^2\g\|_{B^{s-1}_{2,\infty}} \dd\tau
\nonumber\\&\leq C t\f(\f\|v \pa_xv\g\|_{L_t^\infty B^{s-1}_{2,\infty}}+\alpha\f\|(\partial_xv)^2\g\|_{L_t^\infty B^{s-1}_{2,\infty}}\g)\nonumber
\\&\leq C t(1+\alpha)\|v\|^2_{L_t^\infty B^s_{2,\infty}}\leq C t\|v_0\|^2_{B_{2,\infty}^{s}}.
\end{align}
By leveraging Lemmas \ref{le-pro}-\ref{le-pro1}, Lemma \ref{lemm3}, and reapplying the fundamental theorem of calculus with respect to the time variable, we derive the following estimate for $t\in[0,T]$:
\bbal
&\|v(t)-v_0-t\mathbf{E}_0(\alpha,v_0)\|_{B^{s-2}_{2,\infty}}
\leq \int^t_0\big\|\pa_\tau v-\mathbf{E}_0(\alpha,v_0)\big\|_{B^{s-2}_{2,\infty}} \dd\tau
\\&\leq \int^t_0\f(3\|v\pa_xv-v_0\pa_xv_0\|_{B^{s-2}_{2,\infty}} +\alpha^2\f\|\partial^3_x \left(1-\alpha^2 \partial_x^2\right)^{-1}\f(v^2-v_0^2\g)\g\|_{B_{2,\infty}^{s-2}}\g)\dd\tau\\
&\quad+ \int^t_0\frac{\alpha^2}{2}\f\|\partial_x \left(1-\alpha^2 \partial_x^2\right)^{-1} \f((\partial_xv)^2-(\partial_xv_0)^2\g)\g\|_{B^{s-2}_{2,\infty}} \dd\tau\\
&\leq C \int^t_0\f(\|v^2-v_0^2\|_{B_{2,\infty}^{s-1}} +\alpha\f\|(\partial_xv)^2-(\partial_xv_0)^2\g\|_{B^{s-2}_{2,\infty}}\g) \dd\tau
\\&\leq C\int^t_0\|v-v_0\|_{B_{2,\infty}^{s-1}} \dd\tau\les Ct^2,
\end{align*}
where the final step employs the bound provided in \eqref{u1}.
This finishes the proof of the proposition.
\end{proof}

Prior to establishing the proof of Theorem \ref{th1}, it is necessary to introduce smooth, radially symmetric cut-off functions to localize the frequency domain. Specifically,

Before proving theorem 1.1, we need to introduce smooth, radial cut-off functions to localize the frequency region. Precisely,
let $\widehat{\phi}\in \mathcal{C}^\infty_0(\mathbb{R})$ be an even, real-valued and non-negative function defined on $\R$, satisfying the following conditions: 
\begin{numcases}{\widehat{\phi}(\xi)=}
1,&if $|\xi|\leq \frac{1}{4}$,\nonumber\\
0,&if $|\xi|\geq \frac{1}{2}$.\nonumber
\end{numcases}
We define the initial data $u_0(x)$ as
\bbal
u_0(x):=\sum\limits^{\infty}_{n=0}2^{-ns} \phi(x)\cos \big(\frac{17}{12}2^{n}x\big).
\end{align*}
By invoking the definition of the Besov space and considering the support of $\varphi(2^{-j}\cdot)$, we obtain
\bbal
\|u_0\|_{B^{s}_{2,\infty}}&=\sup_{j\geq -1}2^{ js}\|\Delta_{j}u_0\|_{L^2}\\
&=\sup_{j\geq 0}\bi\|\phi(x)\cos \bi(\frac{17}{12}2^{j}x\bi)\bi\|_{L^2}\leq C.
\end{align*}

{\bf Proof of Theorem \ref{th1}.}
Note that
\bbal
&\mathbf{S}^{\alpha_n}_{t}(u_0)=u_0+\underbrace{\mathbf{S}^{\alpha_n}_{t}(u_0)-u_0-t\mathbf{E}_0(\alpha_n,u_0)}_{=:\,\mathbf{I}_1}+t\mathbf{E}_0(\alpha_n,u_0),\\
&\mathbf{S}^{0}_{t}(u_0)=u_0+\underbrace{\mathbf{S}^{0}_{t}(u_0)-u_0-t\mathbf{E}_0(0,u_0)}_{=:\,\mathbf{I}_2}+t\mathbf{E}_0(0,u_0),\\
&\mathbf{E}_0(\alpha_n,u_0)-\mathbf{E}_0(0,u_0)=\underbrace{\alpha^2_n\partial^2_x \left(1-\alpha^2_n \partial_x^2\right)^{-1}(u_0\pa_xu_0)}_{=:\,\mathbf{J}_1}-\underbrace{\frac{\alpha^2_n}{2}\partial_x \left(1-\alpha^2_n \partial_x^2\right)^{-1} (\partial_xu_0)^2}_{=:\,\mathbf{J}_2},
\end{align*}
we deduce from Proposition \ref{pro 3.1}, Lemma \ref{le-pro} and Lemma \ref{lemm3}  that
\bal\label{ql}
&\f\|\mathbf{S}^{\alpha_n}_{t}(u_0)-\mathbf{S}^0_{t}(u_0)\g\|_{B^s_{2,\infty}}\nonumber
\\
\geq&~t2^{ns}\|\De_n[\mathbf{E}_0(\alpha_n,u_0)-\mathbf{E}_0(0,u_0)]\|_{L^2}
-2^{ns}\f\|\De_n\mathbf{I}_1\g\|_{L^2}-2^{ns}\f\|\De_n\mathbf{I}_2\g\|_{L^2}\nonumber\\
\geq&~t2^{ns}\|\De_n\mathbf{J}_1\|_{L^2}-t2^{ns}\|\De_n\mathbf{J}_2\|_{L^2}-2^{2n}\f\|\mathbf{I}_1\g\|_{B^{s-2}_{2,\infty}}-2^{2n}\f\|\mathbf{I}_2\g\|_{B^{s-2}_{2,\infty}}\nonumber \\
\geq&~ t2^{ns}\|\De_n\mathbf{J}_1\|_{L^2}-||\mathbf{J}_2||_{B^s_{2,\infty}}-C2^{2n}t^2\nonumber \\\geq&~ t2^{ns}\|\De_n\mathbf{J}_1\|_{L^2}-C\alpha_n||(\pa_xu_0)^2||_{B^{s-1}_{2,\infty}}-C2^{2n}t^2 \nonumber
\\\geq&~ t2^{ns}\|\De_n\mathbf{J}_1\|_{L^2}-C2^{-n}-Ct^22^{2n}.
\end{align}
Applying Plancherel's identity, we deduce the following equivalence:
\bbal
\f\|\Delta_n[\alpha^2_n\partial^2_x \left(1-\alpha^2_n \partial_x^2\right)^{-1}(u_0\pa_xu_0)]\g\|_{L^2}\approx\|\De_n(u_0\pa_xu_0)\|_{L^2},
\end{align*}
which, in conjunction with \eqref{ql}, leads to the inequality:
\bal\label{ql-1}
\f\|\mathbf{S}^{\alpha_n}_{t}(u_0)-\mathbf{S}^0_{t}(u_0)\g\|_{B^s_{2,\infty}}\geq ct2^{ns}\|\De_{n}(u_0\pa_xu_0)\|_{L^2}-C2^{-n}-Ct^22^{2n}.
\end{align}
Next, we shall estimate the norm $2^{ns}\|\De_{n}(u_0\pa_xu_0)\|_{L^2}$. By Lemma \ref{le3}, we have
\bal\label{ql-10}
2^{ns}\|\De_{n}(u_0\pa_xu_0)\|_{L^2}&=2^{ns}\|u_0\pa_x\De_{n}u_0+[\De_{n},u_0]\pa_xu_0\|_{L^2} \nonumber
\\&\geq \nonumber
2^{{n}s}\|u_0\pa_x\De_{n}u_0\|_{L^2}-2^{{n}s}\|[\De_{n},u_0]\pa_xu_0\|_{L^2}\\&\geq
2^{{n}s}\|u_0\pa_x\De_{n}u_0\|_{L^2}-C||u_0||^2_{B^s_{2,\infty}}.
\end{align}
Since
\bbal
\De_{n}u_0(x)&=2^{-ns} \phi(x)\cos \bi(\frac{17}{12}2^{n}x\bi),
\end{align*}
we find that
\bbal
\pa_x\De_{n}u_0&=2^{-ns} \phi'(x)\cos \bi(\frac{17}{12}2^{n}x\bi)-\frac{17}{12}2^{n}2^{-ns} \phi(x)\sin \bi(\frac{17}{12}2^{n}x\bi).
\end{align*}
Thus,
\bbal
u_0\pa_x\De_{n}u_0&=2^{-ns} u_0(x)\phi'(x)\cos \bi(\frac{17}{12}2^{n}x\bi)-\frac{17}{12}2^{n}2^{-ns} u_0(x)\phi(x)\sin \bi(\frac{17}{12}2^{n}x\bi).
\end{align*}
Observe that $u_0(x)$ is a real-valued and continuous function defined on $\R$. Consequently, there exists some $\delta>0$ such that
\bal\label{ql-20}
&|u_0(x)|\geq \fr{1}{2}|u_0(0)|=\fr{1}{2}\phi(0)\sum\limits^{\infty}_{n=0}2^{-ns}=\frac{2^{s}\phi(0)}{2(2^s-1)}\quad\text{ for any }  x\in B_{\delta}(0).\end{align}
Therefore, it follows from \eqref{ql-20} that
\bbal
\|u_0\pa_x\De_{n}u_0\|_{L^2}
&\geq c2^{n}2^{-ns} \bi\|\phi(x)\sin \bi(\frac{17}{12}2^{n}x\bi)\bi\|_{L^2(B_{\delta}(0))}-C2^{-ns}\bi\| \phi'(x)\phi(x)\cos \bi(\frac{17}{12}2^{n}x\bi)\bi\|_{L^2}\\
&\geq (c2^{n}-C)2^{-ns}.
\end{align*}
By selecting $n$ sufficiently large so that $C<c2^{n-1}$, we derive the following lower bound:
\bal\label{ql-30}
\|u_0\pa_x\De_{n}u_0\|_{L^2}\geq c2^{-n(s-1)}.
\end{align}
Plugging \eqref{ql-10} and \eqref{ql-30} into \eqref{ql-1}, we have for $t\in[0,T]$
\bbal
\f\|\mathbf{S}^{\alpha_n}_{t}(u_0)-\mathbf{S}^0_{t}(u_0)\g\|_{B^s_{2,\infty}}\geq ct2^{n}-Ct-C2^{-n}-Ct^22^{2n}.
\end{align*}
Then, taking large $n>N$ such that $c2^{n}\geq 2C$, we obtain the following inequality for  $t\in[0,T]$
\bal\label{ql-100}
\f\|\mathbf{S}^{\alpha_n}_{t}(u_0)-\mathbf{S}^0_{t}(u_0)\g\|_{B^s_{2,\infty}}\geq ct2^{n}-C2^{-n}-Ct^22^{2n}.
\end{align}
By selecting $t_n=\ep2^{-n}$ with a sufficiently small $\ep$, and choosing $n>N$ large enough such that $C2^{-n}\leq \ep^2$, we obtain from \eqref{ql-100}
\bbal
\f\|\mathbf{S}^{\alpha_n}_{t_n}(u_0)-\mathbf{S}^0_{t_n}(u_0)\g\|_{B^s_{2,\infty}}\geq c\ep-C\ep^2\geq c_0\ep.
\end{align*}
This completes the proof of Theorem \ref{th1}.

\section*{Acknowledgments}
Jianzhong Lu is supported by  National Nature Science Foundation of China (No. 12401165), Hunan Provincial Natural Science Foundation (No. 2024JJ6412), Excellent Youth Project of Hunan Provincial Department of Education (No. 24B0776) and Youth Program of Xiangnan University (No. 2023XJ01). Wei Deng is supported by Science and Technology Research Projects of Jiangxi Provincial Department of Education (No. GJJ2406103).

\section*{Conflict of interest}
The authors declare that they have no conflict of interest.

\section*{Data Availability} Data sharing is not applicable to this article as no new data were created or analyzed in this study.

\end{document}